\documentclass{amsart}
\usepackage{verbatim}
\usepackage{rotating} 
\usepackage{amssymb}
\usepackage{mathrsfs,amsfonts,amsmath,amssymb,color, epsfig,amscd,xy,amsthm,pb-diagram} 

\newtheorem{theorem}{Theorem}[section]
\newtheorem{proposition}[theorem]{Proposition}

\newtheorem{lemma}[theorem]{Lemma}

\newtheorem{corollary}[theorem]{Corollary}

\theoremstyle{plain}

\theoremstyle{remark}

\newtheorem{remark}[theorem]{Remark}

\newcommand{\C}{{\mathbb C}}

\newcommand{\Q}{{\mathbb Q}}

\newcommand{\N}{{\mathbb N}}

\newcommand{\Gm}{\mathbb{G}_{\text{m}}}

\newcommand{\Qbar}{\bar{\Q}}

\DeclareMathOperator{\den}{den}

\DeclareMathOperator{\lcm}{lcm}

\DeclareMathOperator{\Norm}{N}

\makeatletter
\@namedef{subjclassname@2020}{\textup{2020} Mathematics Subject Classification}
\makeatother

\newcommand{\bP}{{\mathbb P}}

\newcommand{\bfx}{{\mathbf x}}

\newcommand{\bfn}{{\mathbf n}}

\newcommand{\bfu}{{\mathbf u}}

\newcommand{\bA}{{\mathbb A}}

\author{Jason P.~Bell}
\address{
Jason P.~Bell\\
Department of Pure Mathematics\\
University of Waterloo\\
Waterloo, Ontario, Canada N2L 3G1}
\email{jpbell@uwaterloo.ca}
\thanks{Jason Bell was supported by NSERC grant RGPIN-2022-02951. Khoa Nguyen was supported
by NSERC grant RGPIN-2018-03770 and CRC tier-2 research stipend 950-231716. Shaoshi Chen was partially supported by the NSFC
grants (No. \ 12271511 and No. \ 11688101), CAS Project for Young Scientists in Basic Research (Grant No.\ YSBR-034),
the Fund of the Youth Innovation Promotion Association (Grant No.\ Y2022001), CAS,
and the National Key Research and Development Project 2020YFA0712300. }

\author{Shaoshi Chen}
\address{
Shaoshi Chen\\
KLMM, Academy of Mathematics and Systems Science\\
Chinese Academy of Sciences\\
Beijing, 100190, China
}
\email{schen@amss.ac.cn}

\author{Khoa D.~Nguyen}
\address{
Khoa D.~Nguyen \\
Department of Mathematics and Statistics\\
University of Calgary\\
AB T2N 1N4, Canada
}
\email{dangkhoa.nguyen@ucalgary.ca}

\author{Umberto Zannier}
\address{
Umberto Zannier\\
Scuola Normale Superiore, Classe di Scienze Matematiche e Naturali, Pisa, Italy
}
\email{umberto.zannier@sns.it}


\begin{document}
	\title[D-finiteness]{D-finiteness, Rationality, and Height III: Multivariate P\'olya-Carlson Dichotomy}
	\begin{abstract}		
	We prove a result that can be seen as an analogue of the P\'olya-Carlson theorem for multivariate D-finite power series with coefficients in $\bar{\mathbb{Q}}$.  In the special case that the coefficients are algebraic integers, our main result says that if  $$F(x_1,\ldots ,x_m)=\sum  f(n_1,\ldots ,n_m)x_1^{n_1}\cdots x_m^{n_m}$$ is a D-finite power series in $m$ variables with algebraic integer coefficients and if the logarithmic Weil height of $f(n_1,\ldots ,n_m)$ is $o(n_1+\cdots +n_m)$, then $F$ is a rational function and, up to scalar multiplication, every irreducible factor of the denominator of $F$ has the form
		$1-\zeta x_1^{q_1}\cdots x_m^{q_m}$
		where $\zeta$ is a root of unity and $q_1,\ldots ,q_m$ are nonnegative integers, not all of which are zero. 	
				\end{abstract}
	
	\maketitle

\section{Introduction}

	In this paper, we continue the investigation, which was begun in the earlier papers \cite{BNZ20_DF, BNZ22_DFII}, of how height restrictions on D-finite power series can force such series to be rational. 
	We let $\N$ denote the set of positive integers and let $\N_0:=\N\cup\{0\}$. 
	For $m\in \N$, we  
	consider
	the ring $K[[x_1,\ldots,x_m]]$
	of power series in $m$ variables over a field $K$ of 
	characteristic $0$.  Let 
	$\bfn=(n_1,\ldots,n_m)\in \N_0^m$ and let
	$\bfx=(x_1,\ldots,x_m)$ be the vector of the indeterminates $x_1,\ldots,x_m$.
	We write
	$\bfx^{\bfn}$ to denote the monomial
	$x_1^{n_1}\ldots x_m^{n_m}$ having the total degree
	$\Vert \bfn\Vert:=n_1+\cdots+n_m$. 
	We also write
	$\displaystyle\frac{\partial^{\Vert \bfn\Vert}}{\partial\bfx^{\bfn}}$
	to denote the operator
	$$\left(\frac{\partial}{\partial x_1}\right)^{n_1}\ldots
	\left(\frac{\partial}{\partial x_m}\right)^{n_m}$$
	on $K[[x_1,\ldots,x_m]]$. A power series $F(\bfx)\in K[[\bfx]]$ is said to be 
	\emph{D-finite} (over $K(\bfx)$) if 
	all the derivatives
	$\displaystyle\frac{\partial^{\Vert \bfn\Vert}F}{\partial\bfx^\bfn}$
	for $\bfn\in \N_0^m$
	span a finite-dimensional vector space over $K(\bfx)$. 
	When using an uppercase letter to denote a power series, we typically use the corresponding lowercase letter to denote its coefficients, for example:
	$$F(\bfx)=\sum_{\bfn\in\N_0^m} f(\bfn)\bfx^{\bfn}\in K[[\bfx]]$$
	with $f:\ \N_0^m\rightarrow K$. If $F(\bfx)\in\Qbar[[\bfx]]$ is D-finite, there exists a number field containing all the coefficients of $F$ 
since all coefficients can be recursively generated by finitely many initial values in $\Qbar$ (see remarks 3.10 in~\cite{Lip89_DF}).
	
	Univariate D-finite series include many important classes of functions, including algebraic power series, hypergeometric series, exponential functions, and many generating functions that arise in algebraic combinatorics (see \cite{Sta80_DF, Sta99_EC2}).  Multivariate D-finite series were introduced by Lipschitz \cite{Lip89_DF} and have again played an important role. In particular, Lipschitz proved that diagonals of multivariate D-finite series are univariate D-finite series and Christol's conjecture \cite{Chr1,Chr2} asserts that all globally bounded G-functions can be obtained as diagonals of multivariate rational power series.  
	
	In the first paper in this series \cite{BNZ20_DF}, we proved a rationality theorem of sorts: if $F(\bfx)\in\Qbar[[\bfx]]$ is D-finite and $h(f(\bfn))=o(\log\Vert\bfn\Vert)$ as $\Vert\bfn\Vert \to\infty$ then
	$F$ is necessarily a rational function, with its denominator of a special form, and where the coefficients $f(\bfn)$ form a finite set.  This improves earlier results of van der Poorten-Shparlinski \cite{vdPS96_OL} and Bell-Chen \cite{BC17_PS}. 
	In the second paper \cite{BNZ22_DFII}, this result was refined in the univariate case.  Here $h$ is the absolute logarithmic Weil height and we refer the reader to \S\ref{sec:prelim} and to the book \cite{BG06_HI} for more details on heights.  In the special case when the coefficients of $F$ are integers, one can interpret this as saying that if $\log\,f(\bfn)/\log\,\Vert\bfn\Vert \to 0$ then $F$ is necessarily rational.
	
	If, however, one puts additional arithmetical constraints on the coefficients, one can sometimes obtain analogous rationality theorems that are significantly stronger.  A key example of this comes from the P\'olya-Carlson theorem \cite{Pol,Car}, which shows that if $G(x)=\sum g(n) x^n$ is an integer power series in one variable then if $h(g(n))=o(\log\,n)$ (equivalently, $G(x)$ has radius of convergence 1), then $G(x)$ is either rational or admits the unit circle as a natural boundary.  In particular, if $G$ is D-finite with integer coefficients and has radius of convergence at least $1$, then it is necessarily rational, as D-finite power series have only finitely many singularities.  On the other hand, an integrality condition on the coefficients is essential in this result; for example, for $k\in \N$, the univariate irrational D-finite series $$\log(1+x^k)=x^k-\frac{x^{2k}}{2}+\frac{x^{3k}}{3}-\cdots$$ has coefficients whose heights grow logarithmically in $n$.
		
		In this paper, we extend the results from \cite{BNZ20_DF} under the weaker condition that $h(f(\bfn))=o(\Vert\bfn\Vert)$ but with arithmetical constraints on the coefficients. For $a\in\Qbar$, we use $\den(a)$ to denote its denominator: this is the smallest $d\in\N$ such that $da$ is an algebraic integer. We prove the following result.
	
	\begin{theorem}\label{thm:main}
	Let $m\in\N$ and let $F(\bfx)=\displaystyle\sum_{\bfn\in\N_0^m} f(\bfn)\bfx^{\bfn}\in\Qbar[[\bfx]]$ be a D-finite power series in $m$ variables with algebraic coefficients. For $N\in\N_0$, put:
	$$h_N=\max\{h(f(\bfn)):\ \Vert\bfn\Vert\leq N\}\ \text{and}\ d_N=\lcm\{\den(f(\bfn)):\ \Vert\bfn\Vert\leq N\}.$$
	Suppose that 
	$h_N=o(N)$ and $\log d_N=o(N)$ as $N\to\infty$ then the following hold:
	\begin{itemize}
		\item [(a)] $F$ is a rational function;
		\item [(b)] up to scalar multiplication, every irreducible factor of the denominator of $F$ has the form:
		$$1-\zeta \bfx^{\bfn}$$
		where $\zeta$ is a root of unity and $\bfn\in\N_0^m\setminus\{0\}$.
	\end{itemize}
	\end{theorem}

 	\begin{remark}
 	Consider the earlier transcendental example:
 	$$\log(1+x^k)=x^k-\frac{x^{2k}}{2}+\frac{x^{3k}}{3}-\cdots=:\sum_{n}a_nx^n$$
 	where $k\in \N$. Put $d_N=\lcm\{\den(a_n):\ n\leq N\}$ then by the Prime Number Theorem, $\displaystyle \log d_N\sim N/k$ as $N\to\infty$. Since $k$ can be arbitrarily large, the condition $\log d_N=o(N)$ in Theorem~\ref{thm:main} is optimal.
 	\end{remark}

 	\begin{corollary}\label{cor:main}
	Let $K$ be a number field and let $S$ be a finite set of places of $K$ containing all the archimedean ones.  Let $m\in\N$ and let 
	$$F(\bfx)=\sum_{\bfn\in\N_0^m} f(\bfn)\bfx^{\bfn}\in K[[\bfx]]$$ be a D-finite power series in $m$ variables with $S$-integer coefficients. Suppose that $h(f(\bfn))=o(\Vert\bfn\Vert)$ as $\Vert\bfn\Vert\to\infty$ then $F$ is a rational function with denominator of the form given in Theorem~\ref{thm:main}(b).	
 	\end{corollary}
 	 We note, in fact, that one can view Corollary \ref{cor:main} as being a multivariate version of the P\'olya-Carlson theorem for D-finite series.  In particular, if one takes $K=\mathbb{Q}$ and $S$ to be the archimedean place then it says that if $F(\bfx)\in \mathbb{Z}[[x_1,\ldots ,x_m]]$ is a D-finite power series in $m$ variables that converges in the open unit polydisc then $F$ is a rational function with denominator of the form given in Theorem~\ref{thm:main}(b).
	 In general, there do exist multivariate analogues of the P\'olya-Carlson theorem (see, for example, Straube \cite{Str}), but they necessarily impose significantly stronger conditions on the region of convergence than in the univariate case.  For example, the bivariate series $\sum_n x^n y^{n!}$ is irrational and converges on the unit polydisc.  For this reason, an additional assumption, such as D-finiteness, is necessary for showing that convergence on the polydisc is sufficient to obtain rationality.

	For a tuple $(\alpha_1,\ldots,\alpha_m)$ of algebraic numbers, we let $H(\alpha_1,\ldots,\alpha_m)$ denote the height of the affine point $(\alpha_1,\ldots,\alpha_m)\in \bA^m(\Qbar)$; in other words, 
	\begin{equation}
	\label{eq:H1} H(\alpha_1,\ldots ,\alpha_m):=H([\alpha_1:\cdots:\alpha_m:1]).
	\end{equation}
	  We then let
 	\begin{equation}
	\label{eq:H2}
	h(\alpha_1,\ldots,\alpha_n)=\log H(\alpha_1,\ldots,\alpha_n).\end{equation}
 	One can reinterpret the main result of \cite{BNZ20_DF} as saying that if $$F(\bfx)=\sum_{\bfn\in\N_0^m} f(\bfn)\bfx^{\bfn}\in\Qbar[[\bfx]]$$ is D-finite and 
 	$h(f(\bfn):\ \Vert\bfn\Vert\leq N)=o(\log N)$ then 
 	$h(f(\bfn):\ \Vert\bfn\Vert\leq N)=O(1)$. This suggests, in a certain sense, that the function $\log N$ is a ``boundary''  function for the growth of
 	$h(f(\bfn):\ \Vert\bfn\Vert\leq N)$: if this latter quantity is much smaller than $\log N$ as $N\to\infty$ then it falls into the next level given in the above-stated dichotomy; namely it is $O(1)$. 
	
 	It follows from Theorem~\ref{thm:main} that a linear function in $N$ is the level right above the function $\log N$ for the growth of $h(f(\bfn):\ \Vert\bfn\Vert\leq N)$:
 	\begin{corollary}\label{cor:o(N) to O(log N)}
 	Let $m\in \N$ and let $F(\bfx)=\displaystyle\sum_{\bfn\in\N_0^m} f(\bfn)\bfx^{\bfn}\in\Qbar[[\bfx]]$ be a D-finite power series in $m$ variables with algebraic coefficients. If $h(f(\bfn):\ \Vert\bfn\Vert\leq N)=o(N)$ as $N\to\infty$
 	then $h(f(\bfn):\ \Vert\bfn\Vert\leq N)=O(\log N)$ as $N\to\infty$.  	\end{corollary}
 	
 	One can regard Corollary~\ref{cor:o(N) to O(log N)} and the discussion before it as variants of results and problems discussed in \cite[Section~5]{BNZ22_DFII}. It is expected that the next level above the linear function in $N$ is
 	the function $N\log N$: if $h(f(\bfn):\ \Vert\bfn\Vert\leq N)=o(N \log N)$ then
 $h(f(\bfn):\ \Vert\bfn\Vert\leq N)=O(N)$. However even the univariate case of this remains a long standing open problem in the theory of Siegel E-functions  (see \cite{Sie29_U}, \cite[pp.~11--12]{Beu08_E}, \cite[Section~2.1]{FR22_OS} and\cite[\S4]{BNZ22_DFII} for more details).
 
 The paper \cite{BNZ20_DF} has motivated several results in both number theory and dynamics, and it is our hope that this paper will similarly motivate results where one imposes additional arithmetical constraints. In arithmetic dynamics, this work motivated a height gap conjecture in work of Bell, Hu, Ghioca, and Satriano  \cite{BHS20_HG,BGS21_DU} concerning the quantity
	$\limsup_{n\to\infty} h(f(\Phi^n(x)))/\log n$,
	where $\Phi:\ X\dashrightarrow X$ and $f:\ X\dashrightarrow\bP^1$ are rational maps
	on quasi-projective varieties defined over $\Qbar$ and $x\in X(\Qbar)$ is such that
	the forward orbit of $x$ under $\Phi$ is well defined and avoids the indeterminacy locus of $f$.  It would be interesting if there were an arithmetic-dynamical height gap result analogous to that of Corollary \ref{cor:o(N) to O(log N)}, which holds for maps under the additional constraint that $\{f(\Phi^n(x))\colon n\ge 0\}$ now takes values in a ring of $S$-integers of a number field, where $S$ is a finite set of places that includes all archimedean ones. In the theory of Mahler functions, Adamczewski, Bell, and Smertnig \cite{ABS20_AH} provide a complete classification of the possible height growth of the coefficients of a Mahler function.  In fact, one can show from the recurrences satisfied by Mahler functions that their coefficients lie in a finitely generated ring and hence these results are close in spirit to the arithmetical constraints we consider in this paper. We note that \cite{BNZ20_DF}, together with earlier work of Perelli and Zannier \cite{PZ84_OR,Zan96_OP}, motivated \cite{BN21_AA}, in which a function field analogue of a conjecture of Ruzsa is established and, in part, helped to inspire Dimitrov's spectacular solution \cite{Dim19_AP} of the Schinzel-Zassenhauss conjecture from the 1960s. More recently, the paper \cite{BNZ22_DFII} motivates a more general criterion for the univariate P\'olya-Carlson dichotomy \cite{BGNS23_AG} with applications to the Artin-Mazur zeta function in certain algebraic dynamical systems  \cite{BC18_DO,BGNS23_AG,BCH}.

The outline of this paper is as follows.  In \S\ref{sec:prelim} we give the basic background on heights and show how the corollaries stated in the introduction quickly follow from Theorem \ref{thm:main} and in \S\ref{sec:thm} we prove our main result.

\section{Some preliminary results and proof of the corollaries}\label{sec:prelim}
In this section, we give a brief overview of the height machinery for later use and we prove Corollary \ref{cor:main} and  Corollary \ref{cor:o(N) to O(log N)} under the assumption that Theorem \ref{thm:main} holds.  In addition we prove the univariate case of Theorem \ref{thm:main}.

For a number field $K$,
	we let $M_K^\infty$ denote the set of archimedean places (equivalence classes of archimedean absolute values) of $K$ and we let 
	$M_K^0$ denote the set of finite places.  
  	We write $M_K=M_K^\infty\cup M_K^0$.
	 For every place
  	$w\in M_K$, let $K_w$ denote the completion of 
  	$K$ with respect to $w$ and we let
  	$d(w)$ denote the quantity $[K_w:\Q_v]$ where $v$ is the restriction of
  	$w$ to $\Q$. 
	We now follow the treatment given in \cite[Chapter~1]{BG06_HI}:
  	for every $w\in M_K$ with restriction $v$ on $\Q$, we may always take $\vert\cdot \vert_v$ to be either the ordinary Euclidean absolute value or the $p$-adic absolute value for some prime $p$.
  	We can then normalize $\vert \cdot\vert_w$ by defining
  	$$\vert x\vert_w = \vert \Norm_{K_w/\Q_v}(x) \vert_v^{1/[K:\Q]}.$$
	Let $m\in\N$ and $P \in \bP^m(\Qbar)$, let $K$ be a number
  	field such that $P$ has a representative 
  	$\bfu=(u_0,\ldots,u_m)\in K^{m+1}\setminus\{\mathbf 0\}$
  	and define:
  	$$H(P)=\prod_{w\in M_K} \max_{0\leq i\leq m}\vert u_i\vert_w.$$
  	Define $h(P)=\log (H(P))$. 
  	
  	For $\alpha\in \Qbar$, 
  	we write $H(\alpha)=H([\alpha:1])$
  	and $h(\alpha)=\log(H(\alpha))$. More generally, for a tuple
  	$(\alpha_1,\ldots,\alpha_m)$ of algebraic numbers, we define $H(\alpha_1,\ldots,\alpha_m)$ and
  	$h(\alpha_1,\ldots,\alpha_m)$ as in Equations (\ref{eq:H1}) and (\ref{eq:H2}).

\begin{proof}[Proof of Corollary~\ref{cor:main} assuming Theorem~\ref{thm:main}]
We assume the notation appearing in the statement of 
Corollary~\ref{cor:main}. For every $a\in K$, we have that $
\den(a)$ divides
$$\prod_{v\in M_K^0,\ \vert a\vert_v>1} \vert a\vert_v^{[K:\Q]}=\prod_{v\in M_K^0} \max\{\vert a\vert_v^{[K:\Q]},1\}.$$
For $N\in\N_0$, put 
$$d_N=\lcm\{\den(f(\bfn)):\ \Vert\bfn\Vert\leq N\}\qquad \textrm{and}\qquad h_N=\max\{h(f(\bfn)):\ \Vert\bfn\Vert\leq N\}.$$ We have that $d_N$ divides
\begin{align*}
&\prod_{v\in M_K^0} \max\{\max\{\vert f(\bfn)\vert_v^{[K:\Q]},1\}:\ \Vert\bfn\Vert\leq N\}\\
&=\prod_{v\in M_K^0\cap S} \max\{\max\{\vert f(\bfn)\vert_v^{[K:\Q]},1\}:\ \Vert\bfn\Vert\leq N\}\\
&\leq \exp([K:\Q]\vert S\vert h_N)=e^{o(N)}
\end{align*}
since each $f(\bfn)$ is an $S$-integer and $h_N=o(N)$. Therefore the power series $F(\bfx)$ satisfies the conditions in Theorem~\ref{thm:main} and we finish the proof.
\end{proof} 

\begin{proof}[Proof of Corollary~\ref{cor:o(N) to O(log N)} assuming Theorem~\ref{thm:main}] 
For $N\in\N_0$, let $h_N$ and $d_N$ be as in the statement of Theorem~\ref{thm:main}. By elementary estimates as in the above proof of Corollary~\ref{cor:main}, we have that the property $h(f(\bfn):\ \Vert\bfn\Vert\leq N)=o(N)$ as $N\to\infty$ implies the property that $h_N=o(N)$ and $\log d_N=o(N)$ as $N\to\infty$. Therefore $F$ is a rational function of the form:
\begin{equation}\label{eq:F as A/special denominator}
\frac{A(\bfx)}{(1-\zeta_1\bfx^{\bfn_1})\cdots (1-\zeta_k\bfx^{\bfn_k})}
\end{equation}
where $A(\bfx)\in\Qbar[\bfx]$, $k\in\N_0$, $\zeta_i$ is a root of unity, and $\bfn_i\in\N_0^m\setminus\{0\}$ 
for $1\leq i\leq k$. 

Let $K$ be a number field containing the coefficients of $A$ and all the $\zeta_i$'s. There exists a positive constant $C$ such that 
when expanding \eqref{eq:F as A/special denominator} into a power series in $\bfx$, 
we can express $f(\bfn)$ for $\bfn\neq 0$ as a sum of at most $C\Vert\bfn\Vert^k$ many terms from a finite set. Therefore the $f(\bfn)$'s have bounded denominators and for every $v\in M_K^{\infty}$ we have
$\vert f(\bfn)\vert_v\ll\Vert\bfn\Vert^k$ as $\Vert\bfn\Vert\to\infty$. This implies
$h(f(\bfn):\ \Vert\bfn\Vert\leq N)=O(\log N)$ as $N\to\infty$.
\end{proof}

  We conclude this section with a proof of Theorem~\ref{thm:main} in the case of univariate power series using  well-known techniques.
  \begin{proposition}\label{prop:univariate}
  Let $F(x)=\displaystyle\sum_{n=0}^{\infty}f(n)x^n\in\Qbar[[x]]$ be a univariate D-finite power series. For $N\in\N_0$, let $h_N$ and $d_N$ be as in Theorem~\ref{thm:main}. Suppose
  that $h_N=o(N)$ and $\log d_N=o(N)$ as $N\to\infty$. Then $F$ is a rational function and all of its finite poles are roots of unity.
  \end{proposition}
 \begin{proof}
 Let $K$ be a number field such that $F(x)\in K[x]$. Consider the Hankel determinant:
\begin{equation*}
\Delta_n:=\det\begin{pmatrix} a_0 & a_{1} & \cdots & a_{n} \\
 a_{1} & a_{2} &\cdots  & a_{n+1}  \\
 \ldots  \\
 a_{n} & a_{n+1} &\cdots  & a_{2n} \end{pmatrix}
\end{equation*}
for $n\geq 0$. The given condition $h_N=o(N)$ implies that $F$ is convergent in the open unit disc $\vert x\vert_v<1$ of $K_v$ for every $v\in M_K$. Since $F$ is D-finite, it can be extended analytically beyond the open unit disc $\vert x\vert_v<1$
for every $v\in M_K^{\infty}$. By Polya's inequality \cite[Section~2]{BNZ22_DFII}, there exists $r<1$ such that
$$\prod_{v\in M_K^{\infty}} \vert\Delta_n\vert_v<r^{n^2}$$
for all large $n$. Together with the condition $\log d_N=o(N)$, we have that
$d_{2n}^{n+1}\Delta_n$ is an algebraic integer and its norm is less than $1$ when $n$ is large. Therefore $\Delta_n=0$ when $n$ is large. Kronecker's criterion \cite[Section~2]{BNZ22_DFII} implies that $F$ is rational. The assertion on the poles of $F$ follows from \cite[Proposition~3.6]{BNZ20_DF}.
 \end{proof}

\section{Proof of Theorem~\ref{thm:main}}
\label{sec:thm}
In this section, we give a proof of our main result.  
We henceforth let $\mathbb{T}$ denote the unit circle. The following lemma is an easy corollary of Proposition 2.3 (iii) in~\cite{Lip89_DF}, whose proof was not given. For the sake of completeness, we provide the detailed proof of this lemma.
\begin{lemma} Let $m\in\N$, let $K$ be a field of characteristic zero, and let $F(\bfx)=\displaystyle\sum  f(\bfn) \bfx^{\bfn} \in K[[\bfx]]$ be a D-finite power series in $m$ variables.  Then:
\begin{itemize}
\item [(a)] If $\alpha_1,\ldots ,\alpha_{m-1}\in K$ then $F(x_1,\ldots ,x_{m-1},\sum_{i=1}^{m-1} \alpha_i x_i) \in K[[x_1,\ldots ,x_{d-1}]]$ is D-finite.
\item [(b)] If $\beta_1,\ldots ,\beta_{m}\in K$ then $F(\beta_1 t,\ldots ,\beta_m t)\in K[[t]]$ is D-finite.
\end{itemize}
\label{lem:spec}
\end{lemma}
\begin{proof}
Let $G(x_1,\ldots ,x_{m-1})=F(x_1,\ldots ,x_{m-1}, \sum_{i=1}^{m-1} \alpha_i x_i)$.  By symmetry, it suffices to show that $\{D_{x_1}^i G\colon i\ge 0\}$ is linearly dependent over $K(x_1,\ldots ,x_{m-1})$, where $D_{x_i}$ is the operator given by partial differentiation with respect to $x_i$.  For $n\ge 0$ we have:
\begin{equation}\label{eq:D_x1^nG}
D^n_{x_1}G = \left(\sum_{j=0}^n {n\choose j} \alpha_1^{n-j} D_{x_1}^j D_{x_m}^{n-j} F \right)\Biggr\rvert_{\substack(x_1,x_2,\ldots,\alpha_1 x_1+\cdots +\alpha_{m-1} x_{m-1})}.
\end{equation}
Let $G_n=\displaystyle\sum_{j=0}^n {n\choose j} \alpha_1^{n-j} D_{x_1}^j D_{x_d}^{n-j} F$ for $n\ge 0$.
Since $F$ is D-finite, we have a linear dependence of the form
\begin{equation}\label{eq:sum P_iG_i=0}
\sum_{i=0}^N P_i(x_1,\ldots ,x_m) G_i(x_1,\ldots ,x_m)=0
\end{equation}
for some $N\ge 0$, with 
$P_0,\ldots,P_N\in K[x_1,\ldots,x_m]$, not all of which are $0$.  Furthermore, we may assume without loss of generality that $P_0,\ldots ,P_N$ have no non-trivial common factor and hence $x_m-\displaystyle\sum_{i=1}^{m-1} \alpha_i x_i$ does not divide all of $P_0,\ldots ,P_m$.  Let $Q_j(x_1,\ldots ,x_{m-1})=P_j(x_1,\ldots ,x_{m-1}, \sum_{i=1}^{m-1} \alpha_i x_i)$ for $j=0,\ldots,N$ then at least one of the $Q_j$'s is nonzero.  Then by \eqref{eq:D_x1^nG} and \eqref{eq:sum P_iG_i=0} we have the non-trivial linear dependence relation:
$$\sum_{i=0}^N Q_i(x_1,\ldots ,x_{m-1}) D_{x_1}^i G =0.$$
This proves part (a) of the lemma.

We prove part (b) by induction on $m$.  When $m=1$, the result is immediate.  We consider $m\geq 2$ and suppose that the statement is valid for power series in fewer than $m$ variables. The case $\beta_1=\cdots=\beta_m=0$ is trivial, so without loss of generality we may assume that $\beta_1\neq 0$. By part (a), we have 
$F(x_1,\ldots ,x_{m-1}, (\beta_m/\beta_1) x_1)$ is a D-finite power series  in $m-1$ variables.  Then making the substitution 
$x_i\mapsto \beta_i t$ for $i=1,\ldots ,m-1$ and using the induction hypothesis gives the desired result.
\end{proof}

\begin{lemma}\label{lem:C1n^C2}
Let $m$, $F=\sum f(\bfn)\bfx^{\bfn}$, $h_N$, and $d_N$ be as in the statement of Theorem~\ref{thm:main}.  If with $h_N=o(N)$ and $\log d_N=o(N)$, then there exist positive constants $C_1$ and $C_2$ such that the following hold.
\begin{itemize}
\item [(a)] $\vert f(\bfn)\vert \leq C_1\Vert\bfn\Vert^{C_2}$ when $\Vert\bfn\Vert$ is large.
\item [(b)] For every $(\alpha_1,\ldots,\alpha_m)\in\mathbb{T}^m$ and for every positive integer $n$, the coefficient of $t^n$ in
$F(\alpha_1t,\ldots,\alpha_mt)$ has modulus at most $C_1n^{C_2}$.
\end{itemize}
\end{lemma}
\begin{proof}
Let $\sigma$ be the complex conjugate automorphism.  Then both $F$ and the series  
$\sigma(F):=\displaystyle\sum_{\bfn}\sigma(f(\bfn))\bfx^\bfn$ are D-finite. Therefore the Hadamard product
$$G(\bfx)=\sum_{\bfn}\vert f(\bfn)\vert^2\bfx^\bfn$$
is also D-finite \cite{Lip89_DF}. The univariate series
$$G(t,\ldots,t)=:\sum_{n}g(n)t^n$$
is D-finite and satisfies: 
\begin{itemize}
\item $h(g(n))=o(n)$; and
\item $\log\lcm\{\den(g(n)):\ n\leq N\}=o(N)$ thanks to the assumption on $h_N$ and $d_N$ for the original series $F$. 
\end{itemize}
By Proposition~\ref{prop:univariate}, $G(t,\ldots,t)$ is a rational function and its poles are roots of unity. Therefore, along suitable arithmetic progressions, we can express the $g(n)$'s as polynomials in $n$ and this implies 
$$\sum_{n_1+\cdots+n_m=n} \vert f(n_1,\ldots,n_m)\vert^2=g(n) \ll n^{C_3}$$
for some positive constant $C_3$. Therefore $\vert f(n_1,\ldots,n_m)\vert \ll n^{C_3/2}$.
Hence if $(\alpha_1,\ldots,\alpha_m)\in\mathbb{T}^m$, the coefficient of $t^n$ in $F(\alpha_1t,\ldots,\alpha_mt)$ has
modulus at most
$$\sum_{n_1+\cdots+n_m=n} \vert f(n_1,\ldots,n_m)\vert\ll n^{m-1+C_3/2}.$$ The result follows.
\end{proof}

\begin{lemma}\label{lem:F(alpha1t,...) is rational}
Let $m$, $F=\displaystyle\sum f(\bfn)\bfx^{\bfn}$, $h_N$, and $d_N$ be as in the statement of Theorem~\ref{thm:main}. If $h_N=o(N)$ and $\log d_N=o(N)$, then for every tuple of roots of unity $(\zeta_1,\ldots,\zeta_m)$, the univariate series
$F(\zeta_1t,\ldots,\zeta_mt)$ is a rational function and all of its finite poles are roots of unity.
\end{lemma}
\begin{proof}
This uses similar arguments to those in the proof of Lemma~\ref{lem:C1n^C2} and Proposition~\ref{prop:univariate}.
\end{proof}

We now prove a key specialization lemma.
\begin{lemma}
\label{lem:W}
Let $E$ be a subfield of $\C$ and let $F(x_1,\ldots ,x_m)\in E[[x_1,\ldots ,x_m]]$ be D-finite.  Suppose that there exist positive constants $C_1$ and $C_2$ such that for all $(\alpha_1,\ldots ,\alpha_m)\in \mathbb{T}^m$ and for all $n\ge 1$, the coefficient of $t^n$ in
$F(\alpha_1 t,\ldots ,\alpha_m t)$ is at most $C_1 n^{C_2}$ in modulus.  Then there exists a nonzero polynomial $W(x_1,\ldots ,x_m, t)\in E[x_1,\ldots ,x_m,t]$ such that whenever
 $(\alpha_1,\ldots ,\alpha_m)\in \mathbb{T}^m$ is such that $F(\alpha_1t,\ldots ,\alpha_mt) $ is a rational function, the series
 $W(\alpha_1,\ldots ,\alpha_m, t)F(\alpha_1t,\ldots ,\alpha_mt) $ has radius of convergence strictly greater than $1$.
\end{lemma}
\begin{proof}
Let $$G(t)=F(x_1 t,\ldots ,x_m t)\in E[x_1,\ldots,x_m][[t]].$$  
Given $j_1,\ldots ,j_d\in \{1,\ldots ,m\}$, we let $F_{j_1,\ldots ,j_d}(x_1,\ldots ,x_m)$ denote the $d$-th partial derivative
$$\displaystyle\frac{\partial^d}{\partial{x_{j_1}}\cdots \partial{x_{j_d}}} F(x_1,\ldots ,x_m).$$
 
Then 
\begin{equation}\label{eq:G^{(d)}}
G^{(d)}(t) = \sum x_{j_1}x_{j_2}\cdots x_{j_d} F_{j_1,\ldots ,j_d}(x_1 t,\ldots ,x_m t),
\end{equation}
 where the sum runs over all $d$-tuples in $\{1,\ldots ,m\}^d$.  Now since $F$ is  D-finite, there exists some natural number $T$ such that for $i=1,\ldots ,m$, we have
$$Q_{i,T}(x_1,\ldots ,x_m) \frac{\partial^T}{\partial{x_i}^T} F(x_1,\ldots ,x_m) = \sum_{j<T} Q_{i,j}(x_1,\ldots ,x_m) \frac{\partial^j}{\partial{x_i}^j} F(x_1,\ldots ,x_m) $$ for some polynomials $Q_{i,j}\in E[x_1,\ldots,x_m]$ with $Q_{i,T}$ nonzero.  It follows that for every $d$ and every $d$-tuple $j_1,\ldots ,j_d$, we have
$$F_{j_1,\ldots ,j_d}(x_1,\ldots ,x_m) \in \sum_{s\le Tm} \sum_{(i_1,\ldots ,i_s)}  E(x_1,\ldots ,x_m) F_{i_1,\ldots ,i_s}(x_1,\ldots, x_m),$$
where the second sum runs over elements of $\{1,\ldots ,m\}^s$. Combining this with \eqref{eq:G^{(d)}}, we have that for every $d$:
$$G^{(d)}(t) = \sum_{s\le Tm} \sum_{(i_1,\ldots ,i_s)} q_{(i_1,\ldots ,i_s),d}F_{i_1,\ldots ,i_s}(x_1 t,\ldots ,x_m t),$$
for some $q_{(i_1,\ldots ,i_s),d}\in E(x_1,\ldots,x_m,t)$. 
Put $M=\displaystyle\sum_{s\leq Tm}m^s$.  Then the $M+1$ power series $G^{(d)}(t)$ for $0\leq d\leq M$ satisfy a non-trivial linear dependence relation:
$$\sum_{d=0}^{M} s_d(x_1,\ldots ,x_m, t) G^{(d)}(t) = 0,$$ with 
$s_0,\ldots,s_M\in E[x_1,\ldots,x_m,t]$.  Let $L$ be the largest index $d$ for which $s_d$ is a nonzero polynomial.

Let $(\alpha_1,\ldots,\alpha_m)\in\mathbb{T}^m$ such that  $A(t):=F(\alpha_1t,\ldots,\alpha_mt)$ 
is a rational function. First consider the case when $s_L(\alpha_1,\ldots,\alpha_m,t)$
is a nonzero polynomial in $\C[t]$. Then we have
the relation
$$\sum_{d=0}^{L}s_d(\alpha_1,\ldots ,\alpha_m, t) A^{(d)}(t) = 0.$$ 
Thanks to this relation, we have that if $A(t)$ has a pole at 
$t=\alpha$ then the leading polynomial $s_L(\alpha_1,\ldots, \alpha_m,t)$ has  a 
zero at $t=\alpha$.  By the given upper bounds on the 
modulus of the coefficients of $A(t)$, we have that
$A(t)$ converges on the open unit disc and every pole $
\alpha\in \mathbb{T}$
has order at most 
$\lfloor C_2\rfloor+1$. Therefore $s_L(\alpha_1,\ldots,\alpha_m,t)^{\lfloor C_2\rfloor+1}A(t)$ has radius of convergence greater than $1$. The case when
$s_L(\alpha_1,\ldots,\alpha_m,t)$
is the zero polynomial is obvious: the product $s_L(\alpha_1,\ldots,\alpha_m,t)^{\lfloor C_2\rfloor+1}A(t)$ is identically zero. We now take $W(x_1,\ldots,x_m,t)=s_L(x_1,\ldots,x_m,t)^{\lfloor C_2\rfloor+1}$ and finish the proof.
\end{proof}

\begin{corollary}\label{cor:WF(alpha1t,...) is a polynomial}
Let $m$, $F=\displaystyle\sum
f(\bfn)\bfx^{\bfn}$, $h_N$, and $d_N$ be as in the statement of Theorem~\ref{thm:main}. If $h_N=o(N)$ and $\log d_N=o(N)$, then there exists a nonzero polynomial $W(x_1,\ldots,x_m,t)\in \Qbar[x_1,\ldots,x_m,t]$ such that for every tuple of roots of unity $(\alpha_1,\ldots,\alpha_m)$, the series
$W(\alpha_1,\ldots,\alpha_m,t)F(\alpha_1t,\ldots,\alpha_mt)$ is a polynomial.
\end{corollary}
\begin{proof}
Lemma~\ref{lem:C1n^C2} implies that $F$ satisfies the conditions in Lemma~\ref{lem:W} and so there exists a polynomial $W(x_1,\ldots,x_m,t)$ as in the conclusion of the statement of Lemma~\ref{lem:W}. Then for every tuple $(\alpha_1,\ldots,\alpha_m)$ of roots of unity,  $A(t):=W(\alpha_1,\ldots,\alpha_m,t)F(\alpha_1t,\ldots,\alpha_mt)$ is a rational function without any poles in the \emph{closed} unit disc. Since the coefficient of $t^n$ in $A(t)$ has height $o(n)$, the only possible finite poles of $A(t)$ are roots of unity. 
Therefore $A(t)$ must be a polynomial.
\end{proof}

\begin{lemma}\label{lem:h(root)<C}
Let $P(x_1,\ldots,x_m,t)\in\Qbar[x_1,\ldots,x_m,t]$ be nonzero and let $\mathcal{X}$ be a Zariski dense subset of $\C^m$ consisting of $m$-tuples of roots of unity. Then there exist a positive constant $C$ and a Zariski dense subset $\mathcal{Y}$ of $\mathcal{X}$ such that every $(\zeta_1,\ldots,\zeta_m,\alpha)\in \mathcal{Y}\times\Qbar$ with $P(\zeta_1,\ldots,\zeta_m,\alpha)=0$ satisfies $h(\alpha)<C$.
\end{lemma}
\begin{proof}
Write $$P(x_1,\ldots,x_m,t)=\displaystyle\sum_{j=0}^d P_j(x_1,\ldots,x_m)t^j,$$
with $P_d\neq 0$. Let $\mathcal{Y}$ be the set of $(\zeta_1,\ldots,\zeta_m)\in\mathcal{X}$ such that $P_d(\zeta_1,\ldots,\zeta_m)\neq 0$. 
We have an explicit upper bound on the height of every zero $t=\alpha$ of $$P(\zeta_1,\ldots,\zeta_m,t)=\sum_{j=0}^d P_j(\zeta_1,\ldots,\zeta_m)t^j$$ in terms of $d$ and $\max_jh(P_j(\zeta_1,\ldots,\zeta_m))$, see 
\cite[Theorem~1.6.13]{BG06_HI}. Since the $\zeta_i$'s are roots of unity, 
we have a uniform upper bound for $\max_j h(P_j(\zeta_1,\ldots,\zeta_m))$.
\end{proof}

We can now prove the rationality part of our main result.
\begin{proof}[Proof of Theorem \ref{thm:main}(a)]
Let $W(x_1,\ldots,x_m,t)$ be as in the conclusion of Corollary~\ref{cor:WF(alpha1t,...) is a polynomial}. Write:
$$\sum_{j=0}^{\infty} g_jt^j:=G(t):=W(x_1,\ldots,x_m,t)F(x_1t,\ldots,x_mt)\in \Qbar[x_1,\ldots,x_m][[t]]$$
with $g_j\in\Qbar[x_1,\ldots,x_m]$ for every $j$. By using similar arguments to the proof of Lemma~\ref{lem:W}, we have a relation:
$$\sum_{d=0}^L s_d(x_1,\ldots,x_m,t)G^{(d)}(t)=0$$
where $L\in\N_0$, each $s_d\in\Qbar[x_1,\ldots,x_m]$, and  $s_L\neq 0$. This gives rise to a non-trivial linear recurrence relation (with \emph{polynomial} coefficients) among the coefficients $g_n$'s of $G$:
\begin{equation}\label{eq:g_n's with a recurrence relation}
P_R(x_1,\ldots,x_m,n)g_{n+R}+\cdots+P_0(x_1,\ldots,x_m,n)g_n=0
\end{equation}
for $n\geq C$ where $C,R\in\N_0$, $P_0,\ldots,P_R\in \Qbar[x_1,\ldots,x_m,\theta]$, and $P_0P_R\neq 0$. Let $\mathcal{X}\subset \C^m$ be the set of tuples of roots of unity. 
By Lemma~\ref{lem:h(root)<C}, there exists $\mathcal{Y}\subseteq\mathcal{X}$ that is Zariski dense in $\C^m$ and a positive constant $C_1$ such that
$n<C_1$ whenever $P_0(\omega_1,\ldots,\omega_m,n)=0$
with $n\in\N_0$ and $(\omega_1,\ldots,\omega_m)\in\mathcal{Y}$.

Given 
$(\omega_1,\ldots,\omega_m)\in\mathcal{Y}$, since $G(t)$ is a polynomial after the specialization at $(x_1,\ldots,x_m)=(\omega_1,\ldots,\omega_m)$, we have 
$g_n(\omega_1,\ldots,\omega_m)=0$ for all sufficiently large $n$. Let $N:=N(\omega_1,\ldots,\omega_m)$ be the largest index $n$ such that $g_n(\omega_1,\ldots,\omega_m)\neq 0$ if such an $n$ exists; otherwise let $N:=N(\omega_1,\ldots,\omega_m):=-1$. Suppose $N\geq C$ then \eqref{eq:g_n's with a recurrence relation} implies that 
$P_0(\omega_1,\ldots,\omega_m,N)=0$, hence $N\leq C_1$. 

We have proved that for every $n>\max\{C,C_1\}$,  $g_n(\omega_1,\ldots,\omega_m)=0$ for every $(\omega_1,\ldots,\omega_m)\in\mathcal{Y}$. Since $\mathcal{Y}$ is Zariski dense in $\C^m$, we have $g_n=0$ for $n>\max\{C,C_1\}$. Consequently
$$W(x_1,\ldots,x_m,t)F(x_1t,\ldots,x_mt)\in \Qbar[x_1,\ldots,x_m,t].$$ 
Since $W$ is nonzero, there exists a nonzero $\lambda\in\Qbar$ such that $W(x_1,\ldots,x_m,\lambda)$ is a nonzero polynomial in $\Qbar[x_1,\ldots,x_m]$.
Then we replace $x_i$ by $x_i/\lambda$ to have that 
$$A(x_1,\ldots,x_m):=W(x_1/\lambda,\ldots,x_m/\lambda,\lambda)\in\Qbar[x_1,\ldots,x_m]\setminus\{0\}$$
and 
$$A(x_1,\ldots,x_m)F(x_1,\ldots,x_m)\in\Qbar[x_1,\ldots,x_m].$$
This finishes the proof.
\end{proof}

We now prove a result that will quickly give us part (b) of Theorem \ref{thm:main}.
\begin{theorem}\label{thm:denominator}
Let $m\in\N$ and let $F(\bfx)=\sum f(\bfn)\bfx^{\bfn}\in\Qbar[[\bfx]]$ be the power series representation of a rational function. For $N\in\N_0$, put:
$$h_N=\max\{h(\bfn):\ \Vert\bfn\Vert\leq N\}.$$
If $h_N=o(N)$ then up to scalar multiplication every irreducible factor of the denominator of $F$ has the form $1-\zeta\bfx^{\bfn}$ where $\zeta$ is a root of unity and $\bfn\in\N_0^m\setminus\{0\}$.
\end{theorem}
\begin{proof}
The case $m=1$ is given in \cite[Proposition~3.6]{BNZ20_DF} and we will use this case repeatedly by considering $F(\omega_1 t,\ldots,\omega_mt)\in\Qbar[[t]]$ for appropriate tuples $(\omega_1,\ldots,\omega_m)$ of roots of unity as in the earlier arguments.

Since the condition on $h_N$ remains valid when we replace $F$ by its product with a polynomial, we may assume that $F=A/B$ where $A,B\in \Qbar[x_1,\ldots,x_m]$, $B$ is irreducible, and $B$ is not a factor of $A$. By the same arguments as in \cite[Section~3.2]{BNZ20_DF}, we have $B(0,\ldots,0)\neq 0$. Without loss of generality, assume $B(0,\ldots,0)=1$. Let $K$ be a number field and let $S$ be a finite set of places of $K$ containing $M_K^{\infty}$ such that the coefficients of $A$ and $B$ are $S$-integers in $K$. Then it follows that the coefficients of $F$ are $S$-integers.

Let $(\omega_1,\ldots,\omega_n)$ be an $n$-tuple of roots of unity, and let
$$\sum_{n=0}^{\infty}c_nt^n:=F(\omega_1t,\ldots,\omega_mt).$$
Then it follows that $h(c_n)=o(n)$ by using similar estimates as before: for every place $v$ of $K(\omega_1,\ldots,\omega_m)$ lying above a place in $M_K\setminus S$ we have $\vert c_n\vert_v\leq 1$ and for every place $w$ of $K(\omega_1,\ldots,\omega_n)$ lying above a place in $S$ we have $\vert c_n\vert_w=e^{o(n)}$ thanks to the given condition $h_N=o(N)$. Therefore $F(\omega_1t,\ldots,\omega_mt)$ is a rational function and every root of its denominator is a root of unity.

Write $B=1+B_1+\cdots+B_d$ with $d\in\N$, $B_i$ is the homogeneous part of degree $i$ for $1\leq i\leq d$, and $B_d\neq 0$. Then we have:
\begin{equation}\label{eq:B(x1t,...) with homog parts}
B(x_1t,\ldots,x_mt)=1+tB_1(x_1,\ldots,x_m)+\cdots+t^dB_d(x_1,\ldots,x_m).
\end{equation}
First we show that $B(x_1t,\ldots,x_mt)$ is irreducible in $\Qbar[x_1,\ldots,x_m,t]$. Suppose otherwise $B(x_1t,\ldots,x_mt)=PQ$ for non-constant $P,Q\in\Qbar[x_1,\ldots,x_m,t]$. Then there must be $i\in\{1,\ldots,m\}$ such that $\deg_{x_i}(P)>0$; otherwise $P\in \Qbar[t]$ has a zero $\gamma\in \Qbar$, which gives that $B(\gamma x_1,\ldots , \gamma x_m)$ is identically zero contradicting $B(0,\ldots,0)=1$. 
Similarly, $\deg_{x_j}(Q)>0$ for some $j\in\{1,\ldots,m\}$ as well. 
Consequently, there exists a nonzero $\lambda\in\Qbar$ such that
$P(x_1,\ldots,x_m,\lambda)$ and $Q(x_1,\ldots,x_m,\lambda)$ are non-constant polynomials in $\Qbar[x_1,\ldots,x_m]$. Replacing $x_i$ by $x_i/\lambda$ for every $i$, we have the non-trivial factorization
$$B(x_1,\ldots,x_m)=P(x_1/\lambda,\ldots,x_m/\lambda,\lambda)Q(x_1/\lambda,\ldots,x_m/\lambda,\lambda)$$
contradicting the irreducibility of $B(x_1,\ldots,x_m)$ in $\Qbar[x_1,\ldots,x_m]$. 

We have that $B(x_1t,\ldots,x_mt)$ is not a factor of 
$A(x_1t,\ldots,x_mt)$ in the polynomial ring $\Qbar[x_1,\ldots,x_m,t]$; otherwise 
we can specialize $t=1$ and have that $B(x_1,\ldots,x_m)$ 
is a factor of $A(x_1,\ldots,x_m)$ in $\Qbar[x_1,\ldots,x_m]$. 
Hence there exist   $\tilde{A},\tilde{B}
\in\Qbar[x_1,\ldots,x_m,t]$ and $G\in \Qbar[x_1,\ldots ,x_m]$, 
with $G\neq 0$, such that:
\begin{equation}\label{eq:G}
A(x_1t,\ldots,x_mt)\tilde{A}(x_1,\ldots,x_m,t)+B(x_1t,\ldots,x_mt)\tilde{B}(x_1,\ldots,x_m,t)=G.
\end{equation}

Let $\mathcal{Z}\subset \C^{m+1}$ be the zero locus of $B(x_1t,\ldots,x_mt)$ and let $\pi$ denote the projection from $\mathcal{Z}$ to the $(x_1,\ldots,x_m)$ coordinates.
By Equation (\ref{eq:B(x1t,...) with homog parts}), $\pi$ surjects onto the complement of the zero set of the polynomials $B_1,\ldots ,B_d$ in $\C^m$.

 Let $U$ denote the Zariski open set of $\C^m$ defined by $B_dG\neq 0$.  By the above remarks, $U$ is contained in the image of $\pi$. We now let $V=\pi^{-1}(U)\subseteq \mathcal{Z}$ and let $\mathcal{X}$ be the Zariski dense subset of $U$ consisting of
tuples of roots of unity. Since the induced map $V\rightarrow U$ is a surjective finite morphism,  $\pi^{-1}(\mathcal{X})$ is dense 
$V$. For each $(\omega_1,\ldots,\omega_m,\theta)\in \pi^{-1}(\mathcal{X})$, we have the following:
\begin{itemize}
\item $B(\omega_1 t,\ldots,\omega_m t)$ is a polynomial of degree $d$ since $B_d(\omega_1,\ldots,\omega_m)\neq 0$, and $\theta$ is a root of this polynomial,
\item $F(\omega_1t,\ldots,\omega_mt)=\displaystyle\frac{A(\omega_1t,\ldots,\omega_mt)}{B(\omega_1t,\ldots,\omega_mt)}$, and
\item the polynomials $A(\omega_1t,\ldots,\omega_mt)$
and $B(\omega_1t,\ldots,\omega_mt)$ do not have a common root thanks to \eqref{eq:G} and the fact that $G(\omega_1,\ldots,\omega_m)\neq 0$.
\end{itemize}
Therefore $B(\omega_1t,\ldots,\omega_mt)$ is the denominator of $F(\omega_1t,\ldots,\omega_mt)$. Then it follows that the root $\theta$ must be a root of unity.

We have proved that every point in $\pi^{-1}(\mathcal{X})$ is  
a tuple of roots of unity. Since $\pi^{-1}(\mathcal{X})$ is Zariski dense in the hypersurface $\mathcal{Z}$, it follows that $\mathcal{Z}$ is a torsion coset of $\Gm^{m+1}$, see \cite[Chapter~3]{BG06_HI} or \cite[Section~3.2]{BNZ20_DF}. Then we use \eqref{eq:B(x1t,...) with homog parts} and \cite[Chapter~3]{BG06_HI}
to conclude that the defining polynomial $B(x_1t,\ldots,x_mt)$ of $\mathcal{Z}$ must have the form $1-\zeta t^k x_1^{n_1}\cdots x_m^{n_m}$ for some root of unity $\zeta$, $k\in\{1,\ldots,d\}$ and $(n_1,\ldots,n_m)\in\N_0^m$ with $n_1+\cdots+n_m=k$. Consequently, $B(x_1,\ldots,x_m)=1-\zeta x_1^{n_1}\cdots x_m^{n_m}$ and this finishes the proof.
\end{proof}
\begin{proof}[Proof of Theorem \ref{thm:main}(b)] This immediately follows from Theorem \ref{thm:main}(a) and Theorem \ref{thm:denominator}.
\end{proof}


\end{document}